\documentclass[11pt]{amsart}
\usepackage{latexsym}
\usepackage{amssymb}
\usepackage{amsmath}
\usepackage{dsfont}
\usepackage{fancybox,xcolor}
\usepackage{amssymb}
\usepackage[utf8]{inputenc}
\usepackage{amsfonts}
\usepackage{latexsym}
\usepackage{amsmath,systeme}
\usepackage[leqno]{mathtools}
\usepackage{amsthm}
\usepackage{graphicx}
\usepackage{mathrsfs}
\usepackage{enumitem}
\usepackage{hyperref}
\usepackage{enumerate}

\usepackage[english]{babel}

\newtheorem{theorem}{Theorem}
\newtheorem{lemma}{Lemma}

\theoremstyle{definition}

\theoremstyle{remark}

\numberwithin{equation}{section}
\usepackage{mathtools}
\mathtoolsset{showonlyrefs}

\newcommand{\R}{{\mathbb R}}
\newcommand{\N}{{\mathbb N}}

\newcommand{\E}{\mathbb{E}}

\renewcommand{\P}{\mathbb{P}}

\newcommand{\half}{\frac{1}{2}}

\newcommand{\G}{\mathcal{G}}
\newcommand{\Domain}{\mathcal{D}}
\newcommand{\B}{\mathcal{B}}
\newcommand{\C}{{\mathcal C}}
\newcommand{\Ec}{\mathcal{E}}
\newcommand{\Fc}{{\mathcal F}}
\newcommand{\Nc}{{\mathcal N}}

\begin{document}
	
\title[The convex envelope and random geometric graphs]{Finding the convex envelope of a boundary datum using random geometric graphs}

\author[A. Deshayes, N. Frevenza, A. Miranda, J. D. Rossi]{Aurelia Deshayes, Nicol\'{a}s Frevenza, Alfredo Miranda, Julio D. Rossi}

	 \address{ 
		 Aurelia Deshayes. Univ Paris Est Creteil, Univ Gustave Eiffel, CNRS, LAMA UMR8050, F-94010 Creteil,
France and IRL CNRS IFUMI-2030, Montevideo, Uruguay.
		 \newline
			\texttt{~aurelia.deshayes@u-pec.fr}; 
		\bigskip
		\newline
		\indent Nicol\'{a}s Frevenza. Departamento de M\'{e}todos Cuantitativos, FCEA,
Universidad de la Rep\'{u}blica
Gonzalo Ram\'{i}rez 1926 (11200), Montevideo, Uruguay.
\newline
			\texttt{~nicolas.frevenza@fcea.edu.uy}; 
		\bigskip
		\newline
		 \indent Alfredo Miranda. Departamento de Matem\'{a}tica, FCEyN, Universidad  de Buenos Aires, Pabell\'{o}n I, Ciudad Universitaria (1428), Buenos Aires, Argentina.   
		 \newline
		 \texttt{~amiranda@dm.uba.ar};	
		 \bigskip
		\newline
		 \indent Julio D. Rossi. Departamento de Matem\'{a}ticas y Estadistica, 
		 Universidad Torcuato Di Tella.  Av. Figueroa Alcorta 7350, Buenos Aires, Argentina.
		 \newline
		 \texttt{~julio.rossi@utdt.edu}		
}
	

	
\medskip

\begin{abstract}
	In this paper we approximate the convex envelope of a boundary datum inside
	 a bounded domain in the Euclidean space. We work with a random graph that is obtained as random points with 
	uniform distribution that are connected by proximity ($x\sim y$ when $|x-y|<r$).  
	On the graph we solve an equation (that approximate the first eigenvalue of the Hessian
	of a smooth function) with an exterior datum. Under appropriate 
	assumptions on $r$ we show that the unique solution to the equation in the graph 
	converges to the convex envelope of the boundary datum as the number of points 
	goes to infinity. 
	\end{abstract}

\maketitle

\section{Introduction}

The purpose of this paper is to find the convex
envelope of a boundary datum as the limit of value functions for games played on random graphs.

\subsection{General convexity} We first recall the usual notion of convexity in Euclidean space. 
	Fix a bounded smooth domain $D \subset {\mathbb{R}}^d$.
	A function $u\colon D \to  {\mathbb{R}}$ is said to be convex if, whenever $x,y \in D$ and the segment $[x,y]\coloneqq \{tx+(1-t)y: t \in (0,1)\}$ is contained in $D$, one has
	\begin{equation} \label{convexo-usual}
		u(tx+(1-t)y) \leq tu(x) + (1-t) u(y), \qquad \forall t \in (0,1).
	\end{equation}
		Convex functions play central role in many areas of mathematics, notably in optimization, where they are distinguished by a number of convenient properties. 
		For instance, a (strictly) convex function has no more than one minimum. Even in infinite-dimensional spaces, under suitable additional hypotheses, convex functions continue to satisfy such properties, and as a result, they are the most well-understood functionals in the calculus of variations. In probability theory, Jensen’s inequality shows that applying a convex function to the expectation of a random variable never exceeds the expectation of the function itself. We refer to \cite{Vel} for a general reference on convexity.

	Suppose now that a convex function $f$ is given only outside a bounded domain
	$D \subset \mathbb{R}^d,$ and we want to recover the function inside assuming that it is convex in $D$.
	A natural candidate is to look for the convex envelope inside $D$ using as boundary
	datum $f\colon \partial D \to  {\mathbb{R}}$. This convex envelope is defined as
	\begin{equation} \label{convex-envelope-usual}
		u^* (x) \coloneqq \sup \Big\{v(x) \colon v 
		\mbox{ is convex in $D$ and verifies } v|_{\partial D} \leq f \Big\}.
	\end{equation}
	In terms of a second-order partial differential equation, it was shown in \cite{OS,Ober} that a function $u$ is convex if and only if 
	\[
		\lambda_1 (D^2 u) (x)\coloneqq
		\inf \Big\{
			\langle D^2 u(x) z, z \rangle\colon z\in\mathbb{S}^{d-1}
		\Big\} \geq 0.
	\] 
	Here $\mathbb{S}^{d-1}\coloneqq \{z\in \mathbb{R}^d \colon |z|=1\}$ denotes the $(d-1)-$sphere. 
	Moreover, if $f$ is continuous on $\partial D$ and $D$ is strictly convex, then the convex envelope of $f$ inside $D$ is the unique solution of
	\begin{align}\label{convex-envelope-usual-eq}
	\begin{cases}
			\lambda_1 (D^2 u) (x) = 0 \qquad &x\in D, \\
		~~~u(x) = f (x)  \qquad &x\in  \partial D.
		\end{cases}
	\end{align}
		The equation \eqref{convex-envelope-usual-eq} is understood in the viscosity sense, and the boundary data are attained continuously. We refer to \cite{BlancRossi,HL1,OS,Ober} for proofs and to \cite{CIL,Koike} for the precise definition of being a viscosity solution.

	\subsection{An informal description of the main results} 
	We move one step further and consider a random point cloud inside a bounded domain 
(for simplicity, we take $[0,1]^d$ with $d\ge 2$), where the points are chosen independently and uniformly at random. 
	We are given a function $f$ defined on the sampled points, except at those lying inside a strictly convex open subdomain $D \subset [0,1]^d$. 
	Our aim is to reconstruct the missing values of $f$ on the points inside $D$, under the assumption that the reconstruction is convex in $D$.
	
Moreover, we assign the missing values in such a way that, as the number of sampled points tends to infinity, the reconstruction converges to the continuous optimal extension, namely to the convex envelope of $f\colon \partial D \to  {\mathbb{R}}$ inside $D$.

	To perform this task we look at the point cloud with a graph structure based on proximity: we consider the graph $\G$ whose vertices are the sampled points and whose edges satisfy $x\sim y$ whenever $|x-y|<r$. Thus $\G$ is the proximity graph associated with the point cloud.

Heuristically, for a smooth function $u$ and any direction $z\in\mathbb{S}^{d-1}$, the second-order central difference along $z$ satisfies, for $r$ small,
	\[
	\frac12 u(x+rz) + \frac12 u(x-rz) - u(x) \approx \frac12 \langle D^2 u(x) z, z \rangle r^2.
	\]
Therefore, if we aim to obtain the continuous convex envelope in the limit, we want to take
the minimum among directions $z$ and try to solve
\begin{equation}
\label{ddd}
\inf_{z\in\mathbb{S}^{d-1}} \left\{ \frac12 u(x+rz) + \frac12 u(x-rz) \right \} - u(x) = 0
\end{equation}	
for $x$ a vertex in $D$ and then take $r \to 0$ as the number of points increases. 
	Making such choices at a vertex $x$ of the graph is delicate: one needs sufficiently many points in $B_r(x)$ so that, in a large number of directions, there exist points close to opposite points on $\partial B_r(x)$. 
	One of the main contributions of this paper is to obtain a precise control on how fast $r$ must decay as the number of points increases, in order for this situation to occur asymptotically in the whole graph almost surely.
	Once such a regime is achieved, we consider a function $u$ defined on the graph that solves the discrete analogue of \eqref{ddd} at the vertices inside $D$, while prescribing $u(x)=f(x)$ for $x\notin D$. The resulting equation governs our discrete approximation of the convex envelope and admits an interpretation in terms of a one-player game (a controller), which we describe and analyze in this paper. For related game-theoretic formulations and results, see \cite{BlancRossi,Blanc-Rossi,Lewicka,MPR,PSSW}.

	As precedents in the literature we quote \cite{Arroyo,Bun,Cal1,Cal2,Cal3,Cal4,Cal5}. 
There is a substantial interplay between partial differential equations (PDEs) and graph-based semi-supervised learning: certain PDEs arise as continuum limits of graph-based learning procedures as the amount of data grows.
In previous references the $p-$Laplacian semi-supervised learning was analyzed 
using the stochastic Tug-of-war game interpretation of the $p-$Laplacian in a graph
(the $p-$Laplacian approximation with $p<d$ turns out to be well-posed in the limit of finite labeled data and infinite unlabeled data).

	Now let us comment very briefly on the main ideas used in the proofs. 
First we need some estimates to ensure that the random points with uniform distribution give rise to a proximity graph with nice properties. 
These estimates are based on probabilistic arguments. 
To pass to the limit and recover the convex envelope of the boundary datum we mainly use PDE techniques and the general theory of viscosity solutions \cite{CIL,Koike} in the spirit of \cite{BarSou}.
	A technically delicate point concerns the estimates near the boundary, which are needed to ensure that the boundary datum is taken with continuity. 
	In \cite{Arroyo} it is proved Holder regularity theory for solutions to quite general discrete equations or equivalently discrete stochastic processes on a random graph. However, the hypotheses used there are not well suited for our problem. The optimal asymptotic regularity for our approximations of the convex envelope is left open.

\subsection{Setting and statements of the results} 
\label{subsect-descrip.environment}
Next, let us describe precisely our general setting and state the main results in this paper.

Let $(X_n)_{n\in\N}$ be a sequence of independent uniform random variables on $[0,1]^d$, $d\geq 2$.
Call $\P$ for the product measure in $\Omega = ([0, 1]^d)^\N$ with marginals uniform on $[0, 1]^d$. 

Let $r>0$ and set $\chi_n=\{X_1,\ldots,X_{n}\}$.
Let $\G(\chi_n,r)$ be the graph whose vertices are $\chi_n$ and whose edges are  
\[
\Ec_{n,r} 
=
\Big\{ (x,y)\in \chi_n ^2 \colon x\neq y \text{ and } |x-y|<r \Big\},
\]
where $|\cdot |$ is the Euclidean norm.
When $(x, y)\in \Ec_{n,r}$ we write $x\sim y$ and denote the neighborhood of $x$ as $\Nc_x =\{y\in \chi_n \colon y\sim x\}$. 

We are interested in the case where $n\to \infty$, so, we need to scale the radius to have non-trivial behaviour.
Let $(r_n)_{n\in\N}$ be a sequence of positive numbers decreasing towards $0$ .
Call $\G_n$ for the graph $\G(\chi_n,r_n)$.
The graph $\G_n$ is a simple (but very interesting) model of a random geometric graph.
Depending on the relation between $n$ and $r_n$ the graph presents different behaviors. 
We will work in the \emph{superconnectivity} regime.
We expand about this concept in Section \ref{section.graph}.
A classical reference for random geometric graphs is \cite{penrose}.

The set of vertices $\chi_{n+1}$ is coupled with $\chi_n$ in such way that $\chi_{n}\subset \chi_{n+1}$, that is, when $n$ increases by one, we add a new point.
The situation for the edges is different, since the threshold $r_n$ also changes with $n$.
Hence, an edge of $\Ec_n$ might not be in $\Ec_{n+1}$.

Let $\C_n$ denote the largest connected component of $\G_n$.
Given a vertex $x$, we will consider its neighbors whose distance from $x$ is close to the connection radius $r_n$.
For a sequence $(\delta_n)_{n\in \N}\subset(0,1)$ with $\delta_n \downarrow 0$ and $x\in\chi_n$, we denote by $\Nc_x^{\delta_n}$ the set of the neighbors of $x$ lying in the annulus with radii $(1-\delta_n)r_n$ and $r_n$, namely 
\begin{equation}
\label{vecinos.borde}
\Nc_x^{\delta_n}
\coloneqq
\{
y \in \Nc_x \colon |x-y|>(1-\delta_n)r_n
\}
.
\end{equation}
We will choose $n, r_n, \delta_n$ so that, $\P$-almost surely, $\Nc_x^{\delta_n} \neq \emptyset$ for every $x\in \C_n$ (see Section~\ref{section.graph}).
For the moment, we will keep the notation $\delta_n$; later on, however, we will choose $\delta_n$ explicitly as a function of the graph parameters $n$ and $r_n$.
For convenience, we will occasionally refer to $\Nc_x^{\delta_n}$ as an annulus, even though it is, strictly speaking, the discrete set of neighbors lying in the annular region with radii $(1-\delta_n)r_n$ and $r_n$.

Now, assume that $\Nc_x^{\delta_n} \neq \emptyset$ for all $x\in \C_n$. 
In general, for any $y \in \Nc_x^{\delta_n}$, its reflected point with respect to $x$ is not a sampled point, therefore we need to define an approximate reflected neighbor.
Consider the map $A_x\colon \Nc_x^{\delta_n} \to \Nc_x^{\delta_n}$, defined as $A_x(y) = y_x$, where $y_x$ is given by
\begin{equation}\label{quasi_sym}
y_x 
\coloneqq
\text{argmin}\{ |z+y-2x|^2 \colon z \in \Nc_x^{\delta_n}\}
,
\end{equation}
that is, $y_x$ is the closest point to the reflection of $y$ with respect to $x$.
When the argmin is not unique (which is an event with zero probability), we choose the first argmin in lexicografic order.

Fix $D\subset [0,1]^d$ a strictly convex set.
Let $f\colon D^c\cap[0,1] \to \R$ be a bounded function, and define 
\[
\Domain_n \coloneqq \C_n\cap D,
\qquad 
\B_n \coloneqq \C_n \cap (D^c\cap [0,1]),
\]
where $\C_n$ is the largest connected component of $\G_n$.

Now, we want to describe a game in the random environment given by $\C_n$, that is, where the possible game positions are vertices from a fixed realization of the largest connected component of the random geometric graph $\G_n$.
Then, the space of game sequences and the strategies are defined with respect to the sampled realization.
Thus, if we prove a certain result, for example ``the value of the game verifies some equation", it will hold almost surely with respect to $\P$.
Such a point of view, where the random environment is fixed, is known as the \emph{quenched} setting in the probability and statistical mechanics literature

The parameters of the game are the number of points, $n$, the sequences $(r_n)_{n\in \N},$ and $(\delta_n)_{n\in \N},$ and the function $f\colon D^c\cap[0,1] \to \R$.
The game has only one player, named $J$. 
Informally, the game works as follows: at initial time, a token is placed at some vertex $x_0\in \Domain_n$, and $J$ picks a vertex $y$ among the neighbors of $x$ in the annulus $\Nc_{x_0}^{\delta_n}$, according to some strategy $S$.
The new game position is $x_1\in \Nc_{x_0}^{\delta_n}$ chosen between $y$ and $y_x$ with equal probability, $1/2$.
At each turn, $J$ repeats this procedure, that is, chooses a neighbor in the annulus from the current position (with some strategy), and the new game state is given with probability 1/2 by this picked point or its approximate symmetric.
Let $\tau$ be the first time at which the game position hits $\B_n$.
The game ends at this final position, $x_\tau\in\B_n$, and $J$ pays $f (x_\tau)$. 
The goal of $J$ is to minimize the expected payoff (we refer to $f$ as the \emph{payoff function}). 

The space of all game sequences starting at $x_0$ is
\[
H_n^{\infty} = \{x_0\} \times \G_n \times \G_n \times \cdots
\]
endowed with the product topology. 
Let $\Fc_n^k$ be the $\sigma$-algebra on $H_n^{\infty}$ generated by cylinders of the form $\{x_0\}\times A_1\times \cdots \times A_k\times \G_n \times \cdots$ where each $A_i\subset \G_n$ is a Borel set.
The family $(\Fc_n^k)_{k\geq 0}$ is a filtration, and let $\Fc_n^{\infty} = \sigma(\cup_{k\geq 0} \Fc_n^k)$ be the smallest $\sigma$-algebra that contains the filtration.

A strategy for $J$ is an adapted sequence of maps $S = (S_{n}^k)_{k\geq 0}$ related to the filtration $(\Fc_n^k)_{k\geq 0}$, such that the next game position is a function of the partial game history, that is, 
\[
S_{n}^k(x_0,\dots,x_k) 
=
x_{k+1} 
\]
where $x_{k+1} \in \Nc_{x_k}^{\delta_n}$.
Then, given the partial history $(x_0,\dots,x_k),$ the next game position is sampled with the probability measure
\begin{equation}
\label{prob-transicion}
\frac{1}{2} \delta_{S_{n}^k(x_0,\dots,x_k)}
+
\frac{1}{2} \delta_{A_{x_k}(S_{n}^k(x_0,\dots,x_k))}.
\end{equation}

Let us fix the starting point $x_0\in\Domain_n$ and the strategy $S$.
Then, by the Kolmogorov's extension theorem, there exists a unique probability measure $P_{S,n}^{x_0}$ on $H_n^{\infty}$ related to the $\sigma$-algebra $\Fc_n^{\infty}$ such that the conditional probability measure with respect to $\Fc_n^k$ is given by \eqref{prob-transicion}. 
We write $E_{S,n }^{x_0}$ for the expectation with respect to $P_{S,n}^{x_0}$.
See \cite[Ch. 7]{Blanc-Rossi} and \cite[Ch. 3]{Lewicka} for more details about the construction of this probability.

The hitting time of the set $\B_n$, that is, 
\begin{equation}
\tau \coloneqq \inf \{k\geq 0  \colon x_k\in \B_n \}
\end{equation}
is an stopping time for the filtration $(\Fc_n^k)_{k\geq 0}$.
Assume that $P_{S,n}^{x_0}(\tau < \infty)=1$ for almost every realization with respect to $\P$ (we will prove this in the next section).
The expected payoff for the player $J$ when the game starts at $x_0$ is $E_{S,n}^{x_0}[ f (x_{\tau})]$.
The game value $u_{n}\colon \Domain_n \to \R$ is defined by
\begin{equation}
\label{game.value}
u_{n}(x_0)
= 
\inf_{S} E_{S,n}^{x_0}[ f (x_{\tau})]
.
\end{equation}
Note that $u_{n}$ depends on $n$, $r_n$, $\delta_n$ (the parameters of the model) and the realization of the graph $\G_n$.
When there is no confusion we summarize these dependences only with the sub-index $n$.

Throughout this article we will assume a specific relationship between the parameters defining the graph and those used to define the game.
As already mentioned, we work in the superconnectivity regime, that is, we assume 
\begin{equation}
\label{condicion.superconnectivity}
nr_n^d/\log n \xrightarrow[n\to\infty]{} +\infty.
\end{equation}
However, in order to establish our results we will need a stronger condition than~\eqref{condicion.superconnectivity}.
To keep the choice of parameters transparent, we will impose
\begin{equation}
\label{condicion.parametros}
\frac{\sqrt{n} r_n^{2d}}{\log n}
 \xrightarrow[n\to\infty]{} 
 +\infty
,
\end{equation}
and we will assume $\delta_n=r_n/\sqrt{\log n}$.
Strictly speaking, one would require the condition $\delta_n=o(r_n)$ in order to prove the convergence of the value function to the limiting equation, together with an additional assumption discussed in Section~\ref{section.graph} (see the remark preceding Lemma~\ref{lemma-todas-las-dir}).

We believe that condition \eqref{condicion.parametros} is not optimal, and that the main result of this paper should remain valid throughout the entire superconnectivity regime.

We want to study the asymptotic behavior of the game value function as $n \to +\infty$.
First, we show that the game value satisfies an equation, called {\it Dynamic Programming Principle} (DPP)
in the game theory literature, and that this equation has an unique solution.

\begin{theorem}
\label{game.value.holds.DPP}
Let $(r_n)_{n\in\mathbb{N}}$ and $(\delta_n)_{n\in\mathbb{N}}$ be sequences satisfying the condition \eqref{condicion.parametros}.
Then, $\P$-almost surely, for $n$ sufficiently large, the game value $u_n\colon \C_n\to \R$ defined in \eqref{game.value} is the unique solution of the following Dynamic Programming Principle (DPP)
\begin{equation}
\label{DPP.intro}
\begin{cases}
\displaystyle u(x)=\min\limits_{y\in \Nc_x^{\delta_n}} 
\left(
\frac{1}{2} 
u(y) 
+
\frac{1}{2}
u(y_x) 
\right)
& x\in \Domain_n \\
u(x)=f(x)
& x\in \B_n
\end{cases} 
.
\end{equation}
\end{theorem}

As we have already mentioned, the convex envelope of a continuous datum $f:\partial D
\mapsto \mathbb{R}$ inside a bounded smooth strictly convex domain is characterized as the 
unique solution to the Dirichlet problem 
\begin{equation}
\label{DP.convex}
\begin{cases}
-\lambda_1[D^2 u](x) = 0 
& x\in D \\
u(x)=f(x)
& x\in \partial D
\end{cases} 
.
\end{equation}
This problem \eqref{DP.convex} has a unique solution in viscosity sense and a comparison principle holds, see \cite{OS,Ober}. 

The game value $u_n$ is defined on $\mathcal{C}_n$, but in order to state our main result we need to extend $u_n$ to the whole domain $D$.
To this end, define $T_n \colon D \to \chi_n$ by setting $T_n(x)=X_i$, where $X_i\in \chi_n$ is the closest point of $\chi_n$ to $x$. 
Ties are broken according to any fixed rule, for instance the lexicographic order. 
We then define $\tilde u_n\colon D \to \mathbb{R}$ by
\[
\tilde u_n(x) \coloneqq u_n\bigl(T_n(x)\bigr),
\]
which extends $u_n$ to the whole domain $D$.
For simplicity, we will also refer to $\tilde{u}_n$ as the game value, when no confusion may arise.

Let us now fix $x\in D$ and consider as a candidate for the limit the function
\[
u(x) = \lim_{n\to\infty} \tilde u_n(x).
\]

Our main result reads as follows.

\begin{theorem}
\label{game.value.convergence}
Let $(r_n)_{n\in\mathbb{N}}$ be a sequence satisfying condition \eqref{condicion.parametros} and let $f\colon D\to \R$ continuous.
Let $u\colon D\to \R$ be the unique solution (in viscosity sense) to the problem \eqref{DP.convex}.
Then, $\P$-almost surely the sequence of the game values $(\tilde{u}_n)_{n\in \N}$ converges uniformly to $u$ as $n \to \infty$.
\end{theorem}

The limit of the sequence of game values is independent of the game parameters, except through the payoff function $f$.

\section{Some properties of \texorpdfstring{$\G_n$}{Gn}}
\label{section.graph}

The classical model for random graphs is the Erd\H{o}s--Rényi model, in which any two vertices are connected with probability $p$, independently of all other pairs. 
Hence, the presence of an edge $x\sim z$ is independent of whether $x\sim y$ or $y\sim z$. 
This is not the case for the random geometric graph $\G_n$ introduced in the previous section.
In the geometric setting, when $x\sim y$ and $y\sim z$ (that is, $x$ is close to $y$ and $y$ is close to $z$) then $x$ will also be relatively close to $z$, which provides information about the probability of an edge between $x$ and $z$.

The geometry of $\G_n$ exhibits different behaviors depending on its parameters (recall that the game description involves $\delta_n$, while the graph $\G_n$ itself depends only on $n$ and $r_n$). 
In general, $\G_n$ is not necessarily connected; however, depending on the sequence $(r_n)_{n\in\N}$, it can be shown that the graph becomes connected with probability one as $n$ goes to infinity. 
It is straightforward to see that the expected degree of a vertex is of order $nr_n^d$. Depending on the asymptotic behavior of this quantity, different limit regimes arise for the size of the largest connected component $\C_n$  and for the connectivity properties of $\G_n$. 
We refer to the case $nr_n^d\to\infty$ as the \emph{dense} regime (that is $\C_n$ is dense in $[0,1]^d$). 
There are several special cases within the dense regime, depending on the expected number of isolated vertices. 
We are particularly interested in the \emph{superconnectivity} regime, corresponding to $nr_n^d/\log n \to \infty$, in which the graph is asymptotically connected with probability one. 
In what follows, we work within the \emph{superconnectivity} regime, while imposing additional conditions on $r_n$ and the other parameters of the game. 
For a comprehensive discussion of these regimes and their connectivity properties, see, for instance, \cite{penrose}.

Recall that $\Domain_n = \C_n \cap D$, where $D$ is a strictly convex open subset of $[0,1]^d$. 
Without loss of generality, we may assume that the distance between $D$ and the complement of $[0,1]^d$ is greater than $r_n$ (which holds for $n$ sufficiently large).

First, we need to prove that, almost surely, $\Nc_{x}^{\delta_n} \neq \emptyset$ for all $x\in\Domain_n$ when $n$ is large, since it is necessary to have an approximate reflected point for every vertex in the annulus. 
Moreover, the proof of Theorem~\ref{game.value.convergence} relies on the fact that, for a regular function, it is possible to estimate the second derivative in the direction of the gradient by means of a special average over a neighborhood. 
This requires that, for every $x \in \Domain_n$, there exist points in $\Nc_{x}^{\delta_n}$ in all directions. 
This section will show that the game parameters $n$, $r_n$, and $\delta_n$ can be chosen so that the two conditions above are satisfied. 
To this end, we will use a version of Talagrand's concentration inequality for empirical processes, refined by Bousquet~\cite{Bousquet2002}.

Let us clarify what we mean by saying that there are points in ``all directions" of the annulus. 
Given $r>0$ and $0<\delta<1$, let $S(r,\delta)$ denote the annulus with inner and outer radii $(1-\delta)r$ and $r$, respectively, centered at the origin:
\begin{equation}
S(r,\delta)
\coloneqq
\{z\in \R^d : r(1-\delta)<|z|<r\}.
\end{equation}
Let $\alpha>0$.  
We define the positive cone $K(\alpha)$ by
\begin{equation}
K(\alpha) 
\coloneqq
\{z = (z_1,\dots,z_d) \in \R^d : \sqrt{z_1^2 + \dots + z_{d-1}^2} \leq \alpha z_d\}.
\end{equation}
The set
\begin{equation}
S(r,\delta,\alpha) \coloneqq S(r,\delta) \cap K(\alpha)
\end{equation}
consists of those points of the annulus that lie close to the direction of $e_d$ when $\alpha$ is small (say $\alpha \in (0,1)$).  

When referring to points in all directions, the aim is to ensure that, for sufficiently small parameters $r_n$, $\delta_n$, and $\alpha_n$, and for sufficiently large $n$, the intersection of $\chi_n$ with the family of sets obtained by translating and rotating $S(r_n,\delta_n,\alpha_n)$ is nonempty with high probability.

More precisely, let $\mathcal{A}_n$ denote the collection of sets obtained by applying rotations and $\mathbb{Q}^d$-translations to $S(r_n,\delta_n,\alpha_n)$, where the parameters $r_n$, $\delta_n$, and $\alpha_n$ are the $n-$th elements of the sequences $(r_n)_{n\in\N}$, $(\delta_n)_{n\in\N}$, and $(\alpha_n)_{n\in\N}$, respectively, in such a way that the resulting sets remain contained in $[0,1]^d$.  
Let $\mathcal{A}$ be the class of sets given by the union of the classes $\mathcal{A}_n$, with $n\in\N$.  
Note that the class $\mathcal{A}$ is countable.  
Moreover, $\mathcal{A}$ has finite VC dimension (Vapnik--Chervonenkis dimension, see for instance \cite{BLM-concentration}).

For $S \in \mathcal{A}$, denote by $\nu(S)$ its Lebesgue measure. 
Observe that all elements within $\mathcal{A}_n$ have the same volume. 
It is not difficult to compute that, if  $\delta_n, \alpha_n \to 0$, we have
\[
\nu(S) \asymp c\, r_n^d \delta_n \alpha_n^{d-1} \qquad \text{ for }
S \in \mathcal{A}_n,
\]
where $c>0$ depends only on the dimension $d$. 
For simplicity, we denote $\gamma_n = c r_n^d \delta_n \alpha_n^{d-1}$.
%

In what follows, we assume that the parameters satisfy
\begin{equation}
\label{condicion.lemma1}
\frac{c\sqrt{n} r_n^d \delta_n \alpha_n^{d-1}}{\log n} 
= \frac{\sqrt{n} \gamma_n}{\log n}\geq 2.
\end{equation}

We define our random variable of interest as
\begin{equation}
\label{defZ}
Z = \sup_{S\in \mathcal{A}} \left|\sum_{i=1}^n  \big(\mathds{1}\{X_i\in S\} - \nu(S)\big) \right|
= 
\sup_{S\in \mathcal{A}} \left|\#(\chi_n \cap S) - n\nu(S)\right|.
\end{equation}
Since the class $\mathcal{A}$ is countable, there are no measurability issues.

To bound the deviations of $Z$ from its expectation $\E Z$, we use Bousquet's inequality \cite[Thm.~2.3]{Bousquet2002}, which yields, for all $t \ge 0$,
\begin{equation}
\label{Bousquet}
\P\!\left(Z > \E Z + \sqrt{t}\,\sqrt{2(n\sigma^2 + 2\E Z)} + \frac{t}{3}\right) 
\le 
e^{-t},
\end{equation}
where $\sigma^2$ is any positive real number surch that 
\[
\sigma^2 \ge \sup_{S\in \mathcal{A}} \E\!\left[(\mathds{1}\{X_1\in S\} - \nu(S))^2\right].
\]
Note that we can choose $\sigma^2 \le 1/4$.

It is possible to bound $\E Z$ by $C\sqrt{n}$, with $C$ a constant independent of $n$, using empirical processes techniques. 
To do so, we follow the strategy outlined in \cite[Section~13.3]{BLM-concentration}, which we briefly describe below. 

The first step is to apply the symmetrization lemma (see \cite[Lem.~11.14]{BLM-concentration}), which in our setting yields
\begin{equation}
\label{lema-simetrizacion}
\E Z \le 2\, \E\!\left( \E_{\epsilon}\left( \sup_{S\in\mathcal{A}}
\left|\sum_{i=1}^n \epsilon_i \mathds{1}\{X_i\in S\}\right|\right) 
\,\Big|\, X_1,\dots,X_n \right),
\end{equation}
where $(\epsilon_i)_{i\in\N}$ is an i.i.d. sequence of Rademacher random variables taking values in $\{-1,1\}$ with probabilty $1/2$, independent of $(X_i)_{i\in\N}$. 
The inner expectation in \eqref{lema-simetrizacion} is called the conditional Rademacher average.

This conditional Rademacher average, suitably normalized by $1/\sqrt{n}$, is sub-Gaussian with respect to a pseudo-metric associated with the empirical measure of $\chi_n$. 
This allows one to apply a version of Dudley's inequality \cite[Lem.~13.5]{BLM-concentration} to bound the conditional expectation using the metric entropy, which is itself bounded because the family of indicator functions we are working with, $\mathds{1}\{X_i\in S\}$ for $S\in\mathcal{A}$, is indexed by a countable class of finite VC dimension (being constructed from balls and cones).
Combining these results yields the bound $\E Z \le C\sqrt{n}$ (see \cite[Thm.~13.5]{BLM-concentration}).

Assuming the above bounds on $\sigma^2$ and $\E Z$, Bousquet's inequality \eqref{Bousquet} yields, for all $t \ge 0$,
\begin{align}
& \P\left(Z > C\sqrt{n} + \sqrt{t}\,\sqrt{\frac{n}{2} + 4C\sqrt{n}} + \frac{t}{3} \right)\\
&\le 
\P\left(Z > \E Z + \sqrt{t}\,\sqrt{2(n\sigma^2 + 2\E Z)} + \frac{t}{3}\right) \label{Bousquet2} \\
&\le 
e^{-t}. 
\end{align}
For each $n \ge 2$, we apply the above inequality with $t = n\gamma_n^2 / \log n$, which diverges to $\infty$ and makes the probabilities in \eqref{Bousquet2} summable, in view of \eqref{condicion.lemma1}. 
Hence, by the Borel–Cantelli lemma, there exists $n_0 = n_0(\omega)$, depending on the realization $\omega$, such that for all $n \ge n_0$ we have
\begin{equation}
\begin{array}{rl}
Z &
= 
\displaystyle \sup_{S\in \mathcal{A}} \left|\#(\chi_n \cap S) - n\nu(S)\right| \\
&
\displaystyle 
\le 
C\sqrt{n} + \frac{n\gamma_n\sqrt{1+4C}}{\sqrt{\log n}} + \frac{n\gamma_n^2}{3\log n} \\
&
\displaystyle 
\le 
\frac{C' n\gamma_n}{\sqrt{\log n}}.
\end{array}
\end{equation}
Therefore, for all $n \ge n_0$ and $S \in \mathcal{A}$, we obtain
\begin{equation}
\label{aux_bousquet}
\#(\chi_n \cap S)
\ge 
n\nu(S) - \frac{C' n\gamma_n}{\sqrt{\log n}}.
\end{equation}
However, the right-hand side of \eqref{aux_bousquet} may be negative for certain values of $n$, although it is asymptotically positive (recall that $S$ is fixed). 
For instance, this can occur when $S \in \mathcal{A}_k$ with $k \gg n \ge n_0$. 
Nevertheless, we are interested in ensuring that $\#(\chi_n \cap S) > 0$ for all $S \in \mathcal{A}_n$ and all $n \ge n_0$. 
Thus, it suffices to require that
\begin{equation}
n\nu(S) = n\gamma_n \ge \frac{C' n\gamma_n}{\sqrt{\log n}},
\end{equation}
which holds for $n$ sufficiently large.

This shows that, when condition \eqref{condicion.lemma1} holds, for a fixed realization of the environment, we have that $\Nc_x^{\delta_n} \neq \emptyset$ for $n$ large enough. 
Moreover, the proof shows that $\P$-almost surely there exists $n_0$, depending on the realization, such that for all $n \ge n_0$ there are points in each set $q + S(r_n,\delta_n,\alpha_n) \in \mathcal{A}_n$ with $q \in \mathbb{Q}^d$; that is, each vertex $x \in \C_n$ has neighbors in all directions of $\Nc_x^{\delta_n}$.

We summarize the result of this section in the following lemma.

\begin{lemma}
\label{lemma-todas-las-dir}
Assume condition~\eqref{condicion.lemma1}.
Then $\P-$almost surely there exists $n_0$, depending on the realization, such that for all $n \ge n_0$ there are points in each set of $\mathcal{A}_n$. 
Hence, $\Nc_x^{\delta_n}$ contains points in ``all directions" for every $x \in \C_n$, for $n\ge n_0$. 
In particular, $\Nc_x^{\delta_n} \neq \emptyset$ for all $x \in \C_n$.
\end{lemma}

Note that if 
\[
\sqrt{n} r_n^d / \sqrt{\log n} \xrightarrow[n\to\infty]{} +\infty,
\]
one can choose the sequences $(\delta_n)_{n\in\N}$ and $(\alpha_n)_{n\in\N}$ so that condition \eqref{condicion.lemma1} is satisfied. This shows that $(\alpha_n)_{n\in\N}$ can be chosen accordingly and is not an additional parameter of the game.
However, for the results in the next sections we will need to impose that $\delta_n =o(r_n)$ and $\alpha_n = o(r_n)$ which leads to a stronger condition on the parameters~\eqref{condicion.parametros}.

\section{Proof of Theorem~\ref{game.value.holds.DPP}}

We fix a realization of the random points $(X_n)_{n\in\N}$ within the set of realizations of probability one for which Lemma~\ref{lemma-todas-las-dir} holds.
We assume that the sequences $(r_n)_{n\in \N}$ and $(\delta_n)_{n\in \N}$ satisfy the condition~\eqref{condicion.parametros} and we fix $n\geq n_0$, where $n_0$ is given by Lemma~\ref{lemma-todas-las-dir}, in order to ensure that the game is always well defined and moreover, that there are points in all directions. We will also need the auxiliary sequence $(\alpha_n)_{n\in\N}$ to satisfy $\alpha_n=o(r_n)$.  

In this section, we will show that the game value $u_n$ verifies the DPP defined in \eqref{DPP.intro}.
First, we show that $\tau$, the hitting time of $\B_n$, is finite $P_{S,n}-$almost surely.
Then, we prove that a comparison principle holds, and by applying the Perron method we show that the DPP has a solution.
The uniqueness is a corollary of the comparison principle.
Finally, using the solution of the DPP we prove that the game value satisfies the DPP.

To begin our analysis of the game we first show that we exit $D$ in a finite
number of steps with $P_{S,n}-$probability one.

\begin{lemma}
\label{lema-juego-termina}
Let $\tau$ be the hitting time of $\B_n$.
Suppose that $n\geq n_0,$ where $n_0$ is given by Lemma \ref{lemma-todas-las-dir}.
Then $$P_{S,n}^{x_0}(\tau < \infty)=1,$$ for any initial position $x_0\in\Domain_n$ and strategy $S$.
\end{lemma}

\begin{proof}
Let $S$ be an arbitrary strategy and let $x_0$ be the initial position. 
At the first step, the new position $x_1$ is chosen with probability $1/2$ between
\[
y = S(x_0) \qquad \text{ and } \qquad y_{x_0} = A_{x_0}(y)
.
\]
Let $y'=2x_0 - y$ be the reflection of $y$ with respect to $x_0$.
Note that $|y_{x_0}-y'|< C_d r_n(\alpha_n +\delta_n)$, since $C_d r_n(\alpha_n+\delta_n)$ bounds the diameter sets of the form $q+S(r_n,\delta_n,\alpha_n)$.

Let $H_0$ be the closed half-space determined by the tangent hyperplane to the sphere $\partial B(0,|x_0|)$ at $x_0$, chosen so that it does not contain the origin. 
Note that if $y\notin H_0$, then necessarily $y'\in H_0$. 
If $y\in H_0$ or $y_{x_0}\in H_0$, then, since $x_1$ is chosen uniformly between $y$ and $y_{x_0}$, with probability 1/2 the distance to the origin increases by at least $c r_n^2$ at the first step.
If instead $y\notin H_0$ and $y_{x_0}\notin H_0$, then $y'\in H_0$. 
Moreover, $|y_{x_0}-y'|\leq C_d r_n(\alpha_n+\delta_n)$ and since $\alpha_n=o(r_n)$, we deduce that $y_{x_0}$ lies sufficiently far on the exterior side, in particular, $|y_{x_0}|\geq |x_0|+c'r_n^2$.

Consequently, at the first step the distance to the origin increases by an amount of order $r_n^2$. The same argument applies at each subsequent step. Since the game evolves in a bounded region, it terminates $P_{S,n}^{x_0}-$ almost surely.
\end{proof}

Now, we turn our attention to the fact that the value of the game is characterized as being the unique
solution to the DPP \eqref{DPP.intro}.

A sub-solution of the DPP equation \eqref{DPP.intro} is a function $u\colon \C_n \to \R$ such that
\begin{equation} 
\begin{cases}
\displaystyle u(x)\le\min\limits_{y\in \Nc_x^{\delta_n}} 
\left(
\frac{1}{2} 
u(y) 
+
\frac{1}{2}
u(y_x) 
\right)
& x\in \Domain_n \\
u(x)\le f(x)
& x\in \B_n
\end{cases} 
.
\end{equation}

Similarly, a function is super-solution when it holds the reverse inequalities ($\geq$ instead of $\leq$).

\begin{lemma}[Comparison Principle]
\label{lema-cp} 
Assume that $n\geq n_0$, where $n_0$ is given by Lemma \ref{lemma-todas-las-dir}.
Let $f, g \colon D^c \cap [0,1] \to \R$ be bounded functions with $f \leq g$.
Let $u$ be a subsolution of the DPP \eqref{DPP.intro} with boundary data $f$ on $\B_n$, and let $v$ be a supersolution of the DPP that coincides with $g$ on $\B_n$.
Then $u \leq v$.
\end{lemma}

\begin{proof}
We argue by contradiction. 
Assume that there exists some vertex $x_0\in\C_n$ such that $u(x_0)>v(x_0)$. 
Then,
\begin{equation}
M
=
\max_{y\in \C_n}
(u-v)(y)>0,
\end{equation}
and it is attained at some $x\in \Domain_n$.
Using that $u$ and $v$ are sub and super-solutions of the DPP, we have that,
\begin{align*}
M &
= u(x)-v(x) \\
&\displaystyle  \leq 
\min_{y\in \Nc_x^{\delta_n}} 
\left(
\frac{1}{2} 
u(y) 
+
\frac{1}{2}
u(y_x) 
\right)
-
\min_{y\in \Nc_x^{\delta_n}} 
\left(
\frac{1}{2} 
v(y) 
+
\frac{1}{2}
v(y_x) 
\right)
\\
& \displaystyle 
\leq 
\frac{1}{2} 
u(z) 
+
\frac{1}{2}
u(z_x) 
-
\frac{1}{2} 
v(z) 
-
\frac{1}{2}
v(z_x)
\\
&\displaystyle  =
\frac{1}{2} 
\left(
u(z) 
-
v(z) 
\right)
+
\frac{1}{2}
\left(
u(y_x) 
-
v(y_x)
\right)
\\
& \displaystyle  \leq M
,
\end{align*}
where $z\in \Nc_x^{\delta_n}$ is such that
\[
\min_{y\in \Nc_x^{\delta_n}} 
\left(
\frac{1}{2} 
v(y) 
+
\frac{1}{2}
v(y_x) 
\right)
=
\frac{1}{2} 
v(z) 
+
\frac{1}{2}
v(z_x)
.
\]
It follows that 
\[
u(z) 
-
v(z) 
= M
\qquad 
\text{and}
\qquad
u(z_x) 
-
v(z_x)
=
M
.
\]
Using the same argument given in Lemma \ref{lema-juego-termina}, where with probability 1/2 the distance to a fixed point increases at each step by at least a fixed amount, we have that eventually reachs the set $\B_n$,
Hence, for some $y\in \B_n$ we obtain
\[
u(y)-v(y) 
= 
M
>
0,
\]
which is a contradiction, since we have $$u(y)=f(y)\leq g(y)= v(y)$$
for all $y\in\B_n$.
\end{proof}

\begin{lemma} 
Assume that $n \geq n_0$, where $n_0$ is given by Lemma \ref{lemma-todas-las-dir}.
Then, there exists a unique function $u \colon \C_n\rightarrow \R$ that verifies 
\begin{equation}
\tag{\ref{DPP.intro}}
\label{DPP.sec}
\begin{cases}
\displaystyle  u(x)=\min\limits_{y\in \Nc_x^{\delta_n}} 
\left(
\frac{1}{2} 
u(y) 
+
\frac{1}{2}
u(y_x) 
\right)
& x\in \Domain_n \\
u(x)=f(x)
& x\in \B_n
\end{cases} 
.
\end{equation}
\end{lemma}

\begin{proof}
First, we show the existence. 
Denote by $\mathscr{A}$ the set of bounded functions $v \colon \C_n\rightarrow\R$ that satisfies
\begin{equation}
\label{subDPP}
\begin{cases}
\displaystyle  v(x)
\leq 
\min\limits_{y\in \Nc_x^{\delta_n}} 
\left(
\frac{1}{2} 
v(y) 
+
\frac{1}{2}
v(y_x) 
\right)
& x\in \Domain_n \\
v(x)
\leq 
f(x)
& x\in \B_n
\end{cases} 
.
\end{equation}
Since the constant function $v(x)=- \|f\|_{\infty}$ verifies \eqref{subDPP}, we have that $\mathscr{A}\neq \emptyset$. 
By the Comparison Principle (Lemma \ref{lema-cp}) any function $v\in \mathscr{A}$ is bounded above by  $\|f\|_{\infty}$.
Then, the function $u\colon \C_n \to \R$ defined by 
\begin{equation}
u(x)
=
\sup_{ v \in\mathscr{A} } v(x)
,
\end{equation}
is well-defined since $\mathscr{A} \neq \emptyset$ and it is bounded by $\|f\|_{\infty}$.

Notice that $u$ belongs to $\mathscr{A}$. 
Let us show that $u$ verifies de DPP \eqref{DPP.sec}. 
Assume that there exists $x_0\in \Domain_n$ such that 
\begin{equation}
\label{DPP.exist.aux}
u(x_0)
<
\min_{y\in \Nc_{x_0}^{\delta_n}} 
\left(
\frac{1}{2} 
u(y) 
+
\frac{1}{2}
u(y_x) 
\right)
.
\end{equation}
Let $\gamma>0$ be small enough so that, by adding $\gamma$ to the left-hand side of the last equation, the strict inequality is still satisfied.
Then it is easy to check that $u^*$ defined by
\begin{equation}
u^*(x)=
\begin{cases}
\displaystyle u(x) & x\neq x_0\\
\displaystyle u(x_0)+ \gamma & x=x_0
\end{cases}
,
\end{equation}
belongs to $\mathscr{A}$ and $u^*>u$, which contradicts the definition of $u$. 
In an analogous way, one can prove that $u(x)=f(x)$ for all $x\in\B_n$.

Uniqueness follows from the Comparison Principle proved in Lemma \ref{lema-cp}.
\end{proof}

Now we are ready to prove that the value of the game is given by the solution to the DPP.
	
\begin{lemma}
\label{valor.juego.DPP}
The game value $u_n$ defined in \eqref{game.value} satisfies the DPP \eqref{DPP.sec}.
\end{lemma}
\begin{proof}
Let $v_n$ denote the unique solution of the DPP \eqref{DPP.sec}, and recall that the game value $u_n$ is defined as 
\[
u_n(x)
=
\inf_{S}
E_{S,n}^{x}[f(x_{\tau})],
\]
where the infimum is taken over all admissible strategies $S$.
We now construct a particular strategy $S^*$ recursively.
Given the initial position $x_0=x$ and the first $k$ positions of the game $x_1, \dots, x_k$, let $y_k\in \Nc_{x_k}^{\delta_n}$ be chosen such that
\begin{equation}
\label{aux.min.strategy}
v_n(x_k)
=
\min_{y\in \Nc_{x_k}^{\delta_n}} 
\frac{1}{2} \left( v_n(y) + v_n(A_{x_k}(y) \right) 
=
\frac{1}{2} \left( v_n(y_k) + v_n(A_{x_k}(y_k) \right).
\end{equation}
We then define $S^*(x_0,\dots,x_k)= y_k$.
Consequently, the next game position $x_{k+1}$ is selected as either $y_k$ or $A_{x_k}(y_k),$ each with probability 1/2.

We claim that the sequence $(v_n(x_k))_{k \geq 1}$ is a martingale.
Indeed, 
\begin{align*}
\displaystyle  E_{S^*,n}^x [ v_n(x_k) | x_1,\dots, x_k] 
= 
\frac{1}{2} v_n(y_{k-1})
+
\frac{1}{2} v_n(A_{x_{k-1}}(y_{k-1}))  = v_n(x_{k-1}),
\end{align*}
where the strategy $S^*$ ensures that condition \eqref{aux.min.strategy} holds.

Since the stopping time $\tau$ is bounded (recall that $n$ is fixed), Doob’s optional stopping theorem yields
\[
v_n(x) = v_n(x_0) = E_{S^*,n}^x [v_n(x_{\tau})] = E_{S^*,n}^x [f(x_{\tau})] \geq u_n(x).
\]

Conversely, let $S$ be an arbitrary strategy.
Define 
\[
z_n(x) = E_{S,n}^x [v_n(x_{\tau})]
,
\] 
and denote $y_1 = S(x)\in \Nc_{x}^{\delta_n}$. 
Using the conditional expectation, it follows that
\[
z_n(x) = \frac{1}{2} z_n(y_1) + \frac{1}{2} z_n(A_x(y_1))\geq \min_{y\in \Nc_{x}^{\delta_n}} \frac{1}{2} z_n(y_1) + \frac{1}{2} z_n(A_x(y_1))
.
\]
Hence, $z_n$ is a supersolution of the DPP \eqref{DPP.sec}. 
By the Comparison Principle, we conclude that $z_n\geq v_n,$ and therefore
\[
u_n(x) = \inf_{S} E_{S,n}^x (f(x_{\tau})) \geq v_n (x)
.
\]

Therefore, we conclude that, for every $x \in \Domain_n$,
\[
v_n(x) =  u_n (x). \qedhere
\]
\end{proof}

\begin{proof}[Proof of Theorem~\ref{game.value.holds.DPP}]
If the sequences $(r_n)_{n\in \N}$ and $(\delta_n)_{n\in \N}$ satisfy the condition~\eqref{condicion.parametros} (and if we choose $\alpha_n$ and $\delta_n$ as $o(r_n)$) then Lemma~\ref{lemma-todas-las-dir} holds and, for $\P-$almost every realization of $(X_n)_{n\in\N}$, there exists $n_0$ such that Lemmas~\ref{lema-juego-termina},~\ref{lema-cp} and~\ref{valor.juego.DPP} hold.
\end{proof}

\section{Proof of Theorem \ref{game.value.convergence} }

To prove the next lemma, we require that $\delta_n = o(r_n)$ and $\alpha_n =o(r_n)$ in addition to condition~\eqref{condicion.lemma1}. For that, we will explicitly  choose $\delta_n = r_n/\sqrt{\log n}$ and $\alpha_n = r_n/(\log n)^{1/2(d-1)}$. 
Recall that $n_{0}$ is a random index depending on the particular realization of the sequence $(X_n)_{n\in\mathbb N}$ that generates the graph $\G_n$.

\begin{lemma} \label{JJJ}
Let $\varphi\in C^2(D)\cap C^1(\overline{D})$.
Assume that $n \geq n_0$, where $n_0$ is provided by Lemma \ref{lemma-todas-las-dir}. 
Then, for all sufficiently large $n$,
\begin{equation}
\label{consistent}
\sup_{x \in \C_n} 
\Big|
\min_{y\in \Nc_x^{\delta_n}}
\left(
\frac{1}{2} 
\varphi(y) 
+
\frac{1}{2}
\varphi(y_x) 
- 
\varphi(x)
\right)
-\frac{r_n^2}{2}
\lambda_1[D^2 \varphi](x)
\Big| 
= 
o(r_n^2)
,
\end{equation}
where $\C_n$ denotes the largest connected component of $\G_n$.
\end{lemma}

\begin{proof}
Fix $x\in \C_n$.
For brevity, set
\[
L(x)
= 
\min_{y\in \Nc_x^{\delta_n}}
\left(
\frac{1}{2} 
\varphi(y) 
+
\frac{1}{2}
\varphi(y_x) 
- 
\varphi(x)
\right)
-\frac{r_n^2}{2}
\lambda_1[D^2 \varphi](x)
.
\]

Let $z\in \Nc_x^{\delta_n}$ be a point such that
\[
\min_{y\in \Nc_x^{\delta_n}}
\left(
\frac{1}{2} 
\varphi(y) 
+
\frac{1}{2}
\varphi(y_x) 
\right)
=
\frac{1}{2} 
\varphi(z) 
+
\frac{1}{2}
\varphi(z_x) 
.
\]
Define $z'= 2x-z$, the reflection of $z$ with respect to $x$ within the annulus $\Nc_x^{\delta_n}$.
In general, however, $z'\notin\G_n$, so, we approximate it by $z_x\in \Nc_x^{\delta_n}$.
We assume that both $z_x$ and $z'$ belong to a set of the form $q+S(r_n,\delta_n,\alpha_n)$ for some $q\in \mathbb{Q}^d$.
We have that
\begin{align}
|L(x)|
& = 
\Big|
\frac{1}{2} 
\varphi(z) 
+
\frac{1}{2}
\varphi(z_x) 
- 
\varphi(x)
-\frac{r_n^2}{2}
\lambda_1[D^2 \varphi](x)
\Big|
\\
& =
\Big|
\frac{1}{2} 
\left(
\varphi(z) 
+ 
\varphi(z') 
\right)
+
\frac{1}{2}
\left(
\varphi(z_x) 
-
\varphi(z') 
\right)
- 
\varphi(x)
-\frac{r_n^2}{2}
\lambda_1[D^2 \varphi](x)
\Big| 
\\
& \leq 
\Big|
\frac{1}{2} 
\left(
\varphi(z) 
+ 
\varphi(z') 
\right)
- 
\varphi(x)
-\frac{r_n^2}{2}
\lambda_1[D^2 \varphi](x)
\Big| 
+ 
\frac{1}{2}
\Big|
\varphi(z_x) 
-
\varphi(z') 
\Big|
\\
\label{consistencia.post.cuentas}
&= |L_1(x)| + |L_2(x)|
.
\end{align}

To control $L_1(x)$ in \eqref{consistencia.post.cuentas}, we use the second-order Taylor expansion of $\varphi$ at $x$, which yields
\begin{align}
\label{consistencia.cota.Taylor}
\frac{1}{2} 
\left(
\varphi(z) 
+ 
\varphi(z') 
\right)
- 
\varphi(x)
= 
\frac{1}{2}
\langle
D^2 \varphi(x) (z-x), (z-x)
\rangle
+
o(r_n^2)
.
\end{align}
Using the variational characterization of the smallest eigenvalue 
\[
\lambda_1[D^2 \varphi](x)
= 
\min_{|v|=1}
\langle
D^2 \varphi(x) v, v
\rangle 
,
\]
and choosing $v=(z-x)/|z-x|,$ and noting that $|z-x|\geq (1-\delta_n)r_n$, we obtain
\begin{equation}
\label{consistencia.cota.autovalor}
\frac{r_n^2}{2}\lambda_1[D^2 \varphi](x)
\leq 
\frac{1}{2(1-\delta_n)^2}
\langle
D^2 \varphi(x) (z-x), (z-x)
\rangle
.
\end{equation}

Subtracting \eqref{consistencia.cota.autovalor} from \eqref{consistencia.cota.Taylor} and taking absolute values, we obtain
\begin{align}
\label{consistencia.cota.L1}
|L_1(x)|
&\leq 
4\delta_n
|
\langle
D^2 \varphi(x)(z-x),(z-x)
\rangle
|
+
o(r_n^2)
\leq 
4C_{\varphi}\delta_n r_n^2 + o(r_n^2).
\end{align}

For $L_2(x)$, the mean value theorem gives
\begin{equation}
\label{consistencia.cota.2}
|L_2(x)|
= 
\frac{1}{2}
|
\varphi(z_x) 
-
\varphi(z') 
|
\leq
C_{\varphi}'
|z_x- z'|
\leq
C_{\varphi}' r_n (\alpha_n+\delta_n)
,
\end{equation}
where $C_{\varphi}'$ depends only on $\varphi$ and $d$, and where $r_n \alpha_n$ (up to a dimensional constant) is the diameter of a translation of $S(r_n,\delta_n,\alpha_n)$.

We therefore obtain 
\[
|L(x)|\leq 4 C_\varphi \delta_n r_n^2 + C_\varphi' r_n(\alpha_n+\delta_n) + o(r_n^2).
\]
Choosing $\delta_n = o(r_n)$ and $\alpha_n =o(r_n)$ completes the proof.
\end{proof}

\begin{lemma}
\label{LemaBorde}
Let $u_n\colon \C_n \to \R$ denote the game value and the solution of the DPP.
Given $\eta>0$, there exists $\rho>0$ such that, for every $n$ large enough, for every $x\in\Domain_n$ and $y\in\partial D$ satisfying $|x-y|<\rho$, we have
\[
|u_n(x)-f(y)|<\eta
.
\]  
\end{lemma}

\begin{proof}
Without loss of generality, assume that $y=0\in\partial D$ and that $D\subset\{x\in\R^d : x_d>0\}$. 
Consider the auxiliary function $v:\R^d\rightarrow\R$ defined by 
\[
v(x)=-K\langle x,e_d\rangle +\frac{\eta}{2} |x|^2+f(0)-\frac{\eta}{2},
\]
where $K>0$ will be chosen later. 
Observe that $v$ is a subsolution to 
\[
-\lambda_1(D^2 u)(x)=0 \qquad x\in D,
\]
and we can choose $K$ sufficiently large so that $v(x)\le f(x)$ on $\partial D$.
Notice that here we are using that $D$ is strictly convex.
 
We now show that $v$ is also a subsolution to the DPP on $\C_n,$
\[
\begin{cases}
	\displaystyle u(x)=\min\limits_{y\in \Nc_x^{\delta_n}} 
	\left(
	\frac{1}{2} 
	u(y) 
	+
	\frac{1}{2}
	u(y_x) 
	\right)
	& x\in \Domain_n \\
	u(x)=f(x)
	& x\in \B_n
\end{cases} 
.
\]
Fix $x\in \Domain_n$ and let $z\in\Nc_x^{\delta_n}$ be such that
\[
\min\limits_{y\in \Nc_x^{\delta_n}} 
\left(
\frac{1}{2} 
v(y) 
+
\frac{1}{2}
v(y_x) 
\right)=\frac{1}{2}v(z)+\frac{1}{2}v(y_z).
\] 
Let $z'=2x-z$ denote the reflection of $z$ with respect to $x$ inside the set $x+S(r_n,\delta_n)$. 
Then,
\[
\frac{1}{2}v(z)+\frac{1}{2}v(y_z)
=
\frac{1}{2}\left(v(z)+v(z')\right)+\frac{1}{2}\left(v(y_z)-v(z')\right).
\]
For the first term, we apply the second-order Taylor expantion of $v$ at $x$, obtaining
\[
\frac{1}{2}v(z)+\frac{1}{2}v(z')=v(x)+\eta|z-x|^2\ge v(x)+\eta r_n^2(1-\delta_n)^2.
\]
For the second term,
\[
|v(y_z)-v(z')|<C|y_z-z'| <Cr_n (\alpha_n+\delta_n).
\]
Requiring $\alpha_n = o(r_n),$ we obtain $|v(y_z)-v(z')| = o(r_n^2)$.

Consequently,
\[
\min\limits_{y\in \Nc_x^{\delta_n}} 
\left(
\frac{1}{2} 
v(y) 
+
\frac{1}{2}
v(y_x) 
\right)- v(x)\ge \eta r_n^2(1-\delta_n)^2+o(r_n^2)\ge 0,
\]
for all sufficiently large $n$.
Hence, $v$ is a subsolution of the DPP.

Using the comparison principle for the DPP, we obtain
\[
u_n(x)\ge v(x)
=
-K\langle x,e_d\rangle+\frac{\eta}{2}|x|^2+f(0)-\frac{\eta}{2}\ge -K|x|+f(0)-\frac{\eta}{2}
.
\]
Choose $\rho>0$ such that $K\rho<\frac{\eta}{2}$. 
Then, for $|x|<\rho$, 
\[
u_n(x)\ge f(0)-\eta.
\]

The reverse inequality,
\[
u_n(x)\le f(0)+\eta,
\]
follows analogously by considering the auxiliary function
\[
w(x)=
K\langle x,e_d\rangle -\frac{\eta}{2} |x|^2+f(0)+\frac{\eta}{2}. \qedhere
\]
\end{proof}

Recall that $T_n\colon D \to \chi_n$ is defined by $T_n(x)=X_i$, where $X_i\in\chi_n$ is the closest point to $x$; ties are resolved according to the lexicographic order. Using this map, we extend $u_n$ to the whole domain $D$ by means of the function $\tilde{u}_n \colon D \to \mathbb{R}$ given by
\[
\tilde{u}_n(x) := u_n\bigl(T_n(x)\bigr).
\]

Our candidate for the limit is then the function
\[
u(x) = \lim_{n\to \infty} \tilde{u}_n(x), \qquad x\in D.
\]

We define the random set of sampled points by
\[
\Domain_{\infty} = \bigcup_{n\geq 1}\Domain_{n}.
\]
Note that $\Domain_{\infty}$ is dense and countable $\mathbb{P}$-almost surely.

We will use the previous technical lemma to prove the following result.

\begin{theorem}
	For $\P$-almost sure realization, the function 
	\[
	u(x) = \lim_{n\to \infty} \tilde{u}_n(x)
	\]
	is well-defined and coincides with the unique viscosity solution to \eqref{DP.convex}.
\end{theorem}

\begin{proof}
	We will work on the set of realizations for which Lemma~\ref{lemma-todas-las-dir} holds, which has probability one with respect to $\P$. All subsequent lemmas are valid on this set of realizations.

	Define
	\[
	\underline{u}(x)=\liminf_{n\to \infty} \tilde{u}_n(x) \qquad \text{and} \qquad \overline{u}(x)=\limsup_{n\to \infty} \tilde{u}_n(x) 
	\]
	Clearly, $\underline{u}\le \overline{u}$. 
	We show that $\underline{u}=\overline{u}$. 
	To this end, we prove that $\underline{u}$ is a viscosity supersolution of \eqref{DP.convex}, while $\overline{u}$ is a viscosity subsolution. 
	The comparison principle then yields $\overline{u}\le\underline{u}$, and therefore $\underline{u}=\overline{u}=u$, which is the unique viscosity solution of \eqref{DP.convex}.
	This relies on the existence, uniqueness, and boundary continuity for \eqref{DP.convex}; see \cite{OS,Ober}.	
	
	We first show that $\underline{u}$ is a viscosity supersolution. Let $\varphi\in C^2(D)$ be a test function that touches $\underline{u}$ from below at some  $x_0\in D$, that is, $(\underline{u}-\varphi)(x_0)=0$ and $x_0$ is a strict minimum of $\underline{u}-\varphi$.
	
	We claim that there exists a sequence $(x_n)_{n\in \N}$ such that $x_n\to x_0,$ and for which
	\[
	(\tilde{u}_n -\varphi)(x_n)\le (\tilde{u}_n - \varphi)(y) \quad \text{ for all } y\in\Domain_n.
	\] 
	Indeed, since $\Domain_n$ is finite, we may choose $x_n$ where $\tilde{u}_n-\varphi$ attains its minimum. 
	Note that $x_n\in\Domain_n$ for all $n$.
	To show that $x_n\to x_0$, note that $(x_n)_{n\in\N} \subset \overline{D}$ is bounded, and thus, up to a subsequence, $x_n\to z$ for some $z$. By the definition of $\underline{u}$,
	 \[
	 \lim_{n\to \infty} (\tilde{u}_n-\varphi)(x_n) \geq (\underline{u}-\varphi) (z).
	 \]
Since the sampled points are dense in $D$, there exists $(y_n)_{n\in \N}\subset \Domain_{\infty}$ with $y_n\to x_0$, such that
	\[
	 \lim_{n\to \infty} (\tilde{u}_n -\varphi)(y_n) = (\underline{u}-\varphi) (x_0) =0.
	 \]
	 If $z\neq x_0$, as we have that $(\underline{u}-\varphi)(x_0)=0$ is a strict minimum of $(\underline{u}-\varphi)$, then $(\underline{u}-\varphi) (z) >0$.
	Therefore, for $n$ large,
	\[
	 \min_{y\in \Domain_n} (\tilde{u}_n-\varphi) (y) = (\tilde{u}_n-\varphi)(x_n) \geq \frac12 (\underline{u}-\varphi) (z) >
	 (\tilde{u}_n-\varphi)(y_n),
	 \]
	 a contradiction. 
	 Hence $x_n\to x_0$.
	
	As the minimum of $\tilde{u}_n - \varphi$ is attained at $x_n$, we have
	\begin{equation}
		\varphi(y)-\varphi(x_n)\le \tilde{u}_n (y)- \tilde{u}_n (x_n) \qquad \text{ for all }y\in\Domain_n.
	\end{equation}
	Using the DPP we get
	\begin{align}
	\displaystyle 	
	0=& \min_{y\in \Nc_{x_n}^{\delta_n}} \half \big( \tilde{u}_n(y) + \tilde{u}_n (y_{x_n})-2 \tilde{u}_n(x_n)\big)\\
			\label{111}
	\displaystyle 
	\ge& \min_{y\in \Nc_{x_n}^{\delta_n}} \half \Big[ \big(\varphi(y)+\varphi(y')-2\varphi(x_n)\big)+\big(\varphi(y_{x_n})-\varphi(y')\big)\Big].
	\end{align}
	Here we use again the reflected point $y'=2x-y$. 
	As in Lemma \ref{LemaBorde} we can obtain
		\[
	|\varphi(y_{x_n})-\varphi(y')| = o(r_n^2).
	\]
	Returning to \eqref{111} and applying Taylor expansion, as in Lemma \ref{JJJ}, we obtain
	\begin{align}
	\displaystyle 
	0\ge& \min_{y\in \Nc_{x_n}^{\delta_n}}\half\langle D^2\varphi(x_n)(y-x_n),(y-x_n)\rangle+o(r_n^2)\\
	\displaystyle 
	\ge& \min_{y\in \Nc_{x_n}^{\delta_n}}\half\langle D^2\varphi(x_n)\frac{(y-x_n)}{|y-x_n|},\frac{(y-x_n)}{|y-x_n|}\rangle r_n^2(1-\delta_n)^2+o(r_n^2)\\
	\displaystyle 
	\ge& \half \min_{|v|=1}\langle D^2\varphi(x_n)v,v\rangle r_n^2(1-\delta_n)^2+o(r_n^2).
	\end{align}
	Dividing by $r_n^2$ and letting $n\to\infty$ gives
	\[
	0 \geq \lambda_1(D^2\varphi(x_0)),
	\]
	proving that $\underline{u}$ is a viscosity supersolution. 

	Analogously, one can prove that $\overline{u}$ is a viscosity subsolution. Using Lemma \ref{LemaBorde}, we have
	\[
	f(y)-\eta \ge \tilde{u}_n (x_n) = u_n(x_n) \le f(y)+\eta,
	\]
	for $y\in\partial D$ and $x_n\in\Domain_n$ sufficiently close to $y$. Taking $\liminf$ yields
	\[
	\underline{u}(y)\ge f(y),
	\]
	while taking the $\limsup$ gives
	\[
	\overline{u}(y)\le f(y).
	\]
	Hence
	\[
	\underline{u}(y)\geq f(y) \geq 
	\overline{u}(y), \qquad y \in \partial D.
	\]

	Since $\underline{u}$ is a supersolution and $\overline{u}$ is a subsolution of \eqref{DP.convex}, the comparison principle implies
	\[
	\overline{u}(x)\le\underline{u}(x),\qquad x \in \overline{D}.
	\] 
	Therefore
	\[
	\underline{u}(x)=\overline{u}(x)=u(x), \qquad x \in \overline{D},
	\] 
	and 
	\[
	\liminf_{n\to \infty} \tilde{u}_n(x_n) = \limsup_{n\to \infty} \tilde{u}_n(x_n) = u(x).
	\]
	Hence, the limit exists and the resulting function $u$ is the unique viscosity solution of \eqref{DP.convex}. 
\end{proof}

\section*{Acknowledgments and funding}

N.F. thanks Alejandro Cholaquidis for valuable discussions on concentration inequalities and empirical processes.

This research was partially supported by the Math AmSud Program (grants 21-MATH-04 and 23-MATH-08). 
N.F. was additionally supported by the Ministerio de Educación y Cultura (FVF-061-2021) and the Agencia Nacional de Investigaci\'{o}n e Innovaci\'{o}n (FCE-3-2024-1-181302). 
A.M. and J.D.R. were partially supported by UBACyT 20020160100155BA (Argentina) and CONICET PIP GI No 11220150100036CO (Argentina).
A.D. was partially supported by the CNRS International Research Laboratory
IFUMI-2030 of Uruguay.

\bibliographystyle{plain}


\begin{thebibliography}{15}

\bibitem{Arroyo} 
A.~Arroyo, P.~Blanc and M.~Parviainen,
Krylov--Safonov theory for Pucci-type extremal inequalities on random data clouds,
arXiv:2410.01642 (2024).

\bibitem{BarSou} 
G.~Barles and P.~E.~Souganidis,
Convergence of approximation schemes for fully nonlinear second order equations,
Asymptot. Anal. 4 (1991), 271--283.

\bibitem{BlancRossi} 
P.~Blanc and J.~D.~Rossi,
Games for eigenvalues of the Hessian and concave/convex envelopes,
J. Math. Pures Appl. (9) 127 (2019), 192--215.

\bibitem{Blanc-Rossi}
P.~Blanc and J.~D.~Rossi,
\textit{Game theory and partial differential equations},
De Gruyter Ser. Nonlinear Anal. Appl. 31, De Gruyter, Berlin, 2019.

\bibitem{BLM-concentration}
S.~Boucheron, G.~Lugosi and P.~Massart,
\textit{Concentration inequalities},
Oxford Univ. Press, Oxford, 2013.

\bibitem{Bousquet2002}
O.~Bousquet,
A Bennett concentration inequality and its application to suprema of empirical processes,
C. R. Math. Acad. Sci. Paris 334(6) (2002), 495--500.

\bibitem{Bun} 
L.~Bungert, J.~Calder and T.~Roith,
Uniform convergence rates for Lipschitz learning on graphs,
IMA J. Numer. Anal. 43(4) (2023), 2445--2495.

\bibitem{Cal1} 
J.~Calder and N.~Garc\'ia Trillos,
Improved spectral convergence rates for graph Laplacians on $\varepsilon$-graphs and k-NN graphs,
Appl. Comput. Harmon. Anal. 60 (2022), 123--175.

\bibitem{Cal2} 
J.~Calder, N.~Garc\'ia Trillos and M.~Lewicka,
Lipschitz regularity of graph Laplacians on random data clouds,
SIAM J. Math. Anal. 54(1) (2022), 1169--1222.

\bibitem{Cal3} 
J.~Calder and D.~Slep\v{c}ev,
Properly-weighted graph Laplacian for semi-supervised learning,
Appl. Math. Optim. 82 (2020), 1111--1159.

\bibitem{Cal4} 
J.~Calder,
The game theoretic $p$-Laplacian and semi-supervised learning with few labels,
Nonlinearity 32(1) (2018), 301--330.

\bibitem{Cal5} 
J.~Calder and N.~Drenska,
Consistency of semi-supervised learning, stochastic tug-of-war games, and the $p$-Laplacian,
Active Particles IV: Advances in Theory, Models, and Applications (2024), 1--53.

\bibitem{CIL} 
M.~G.~Crandall, H.~Ishii and P.~L.~Lions,
User's guide to viscosity solutions of second order partial differential equations,
Bull. Amer. Math. Soc. 27 (1992), 1--67.

\bibitem{HL1} 
F.~R.~Harvey and H.~B.~Lawson Jr.,
Dirichlet duality and the nonlinear Dirichlet problem,
Comm. Pure Appl. Math. 62 (2009), 396--443.
		
\bibitem{Koike}  
S.~Koike,
\textit{A beginner's guide to the theory of viscosity solutions},
MSJ Mem. 13, Math. Soc. Japan, Tokyo, 2004.

\bibitem{Lewicka}
M.~Lewicka,
\textit{A course on tug-of-war games with random noise},
Universitext, Springer, Cham, 2020.

\bibitem{MPR} 
J.~J.~Manfredi, M.~Parviainen and J.~D.~Rossi,
On the definition and properties of $p$-harmonious functions,
Ann. Sc. Norm. Super. Pisa Cl. Sci. 11(2) (2012), 215--241.

\bibitem{OS} 
A.~M.~Oberman and L.~Silvestre,
The Dirichlet problem for the convex envelope,
Trans. Amer. Math. Soc. 363(11) (2011), 5871--5886.

\bibitem{Ober} 
A.~M.~Oberman,
The convex envelope is the solution of a nonlinear obstacle problem,
Proc. Amer. Math. Soc. 135(6) (2007), 1689--1694.

\bibitem{PSSW} 
Y.~Peres, O.~Schramm, S.~Sheffield and D.~B.~Wilson,
Tug-of-war and the infinity Laplacian,
J. Amer. Math. Soc. 22(1) (2009), 167--210.

\bibitem{penrose}
M.~Penrose,
\textit{Random geometric graphs},
Oxford Stud. Probab. 5, Oxford Univ. Press, Oxford, 2003.

\bibitem{Vel} 
M.~L.~J.~van~de~Vel,
\textit{Theory of convex structures},
North-Holland, Amsterdam, 1993.



\end{thebibliography}

\end{document}